\documentclass{amsart}
\usepackage{amsmath,amsthm,graphicx}
\newcommand{\ba}{\begin{array}}
\newcommand{\ea}{\end{array}}
\newcommand{\R}{\mathbb R}%
\newcommand{\C}{\mathbb C}%

\newcommand{\lgra}{\longrightarrow}
\newtheorem{Definition}{Definition}[section]
\newtheorem{Theorem}[Definition]{Theorem}
\newtheorem{Lemma}[Definition]{Lemma}
\newtheorem{Proposition}[Definition]{Proposition}
\newtheorem{Corollary}[Definition]{Corollary}

\begin{document}
\title[foliated immersions in a symplectic manifold]{Smooth maps of a foliated manifold in a symplectic manifold}
\author[M. Datta]{ Mahuya Datta }
\address[M. Datta]{Statistics and Mathematics Unit, Indian Statistical Institute,
203, B.T.Road, Calcutta 700 108, India\\ e-mail:
mahuya@isical.ac.in}
\author[R. Islam]{Md. Rabiul Islam}
\thanks{The second author is supported by the
CSIR fellowship}
\address[R. Islam]{Department of Pure Mathematics,
University College of Science, University of Calcutta, 35, P.Barua Sarani, Calcutta 700 019\\
e-mail: rabiulislam.cu@gmail.com} 
 \maketitle
\section{Introduction}
Gromov proves in \cite{gromov} that the immersions of a smooth
manifold $M$ in a symplectic manifold $(N,\sigma)$ inducing a
given closed form $\omega$ on $M$ satisfy the $C^0$-dense
$h$-principle in the space of all continuous maps which pull back
the deRham cohomology class of $\sigma$ onto that of $\omega$. In
this paper we prove a foliated version of this result.

Let $M$ be a smooth manifold with a regular foliation ${\mathcal F}$
and let $\omega$ be a 2-form on $M$ which induces closed forms on
the leaves of ${\mathcal F}$ in the leaf topology. We shall refer to
the induced form on any leaf as the `restriction' of the global form
on it. A smooth map $f:(M,\mathcal{F})\lgra N$ is called a {\em
foliated immersion\/} if $f$ restricts to an immersion on each leaf
of the foliation. Further, if the restriction of $f^*\sigma$ is the
same as the restriction of $\omega$ on each leaf of the foliation
then $f$ is called a {\em foliated symplectic immersion\/}.

If $f$ is a foliated symplectic immersion then the derivative map
$Df$ gives rise to a bundle morphism $F:TM\lgra TN$ which
restricts to a monomorphism on $T\mathcal{F}{\subseteq}TM$ and
satisfies the condition $F^*\sigma = \omega$ on $T\mathcal{F}$. A
natural question is whether the existence of such a bundle map $F$
ensures the existence of a foliated immersion satisfying
$f^*\sigma=\omega$ on ${\mathcal F}$. As we shall see in this
paper, the obstruction to the existence of such an $f$ is only
topological in nature.

Let Symp$^0(T{\mathcal F},TN)$ denote the space of bundle
morphisms $F:TM{\rightarrow}TN$ such that $F$ restricted to
$T{\mathcal F}$ is a monomorphism and $F^*\sigma = \omega$ on
$T\mathcal{F}$. We endow this space with the $C^0$-compact-open
topology.

\begin{Theorem} Let $\omega$ be a foliated closed 2-form on
$(M,{\mathcal F})$ and let $(F_0,f_0):TM\lgra TN$ be a bundle
homomorphism which satisfies the following condition:
\begin{enumerate}\item $f_0$ pulls back the cohomology class of $\sigma$ onto the
foliated cohomology class of  $\omega$;\item $F_0$ is fibrewise
injective on $T\mathcal{F}$ and $F_0^*\sigma=\omega$ on
$T{\mathcal F}$.
\end{enumerate}
If $dim(\mathcal{F}) < dim(N)$, then the map $f_0$ admits a fine
$C^0$ approximation by a foliated symplectic immersion
$f:(M,\mathcal{F},\omega){\rightarrow}(N,\sigma)$ whose
differential $Df:TM{\rightarrow}TN$ is homotopic to $F_0$ in the
space Symp$^0(T{\mathcal F},TN)$.\label{main}\end{Theorem}

In other words, foliated symplectic immersions satisfy the
$C^0$-dense h-principle in the space of continuous maps
$f:(M,\mathcal{F}){\rightarrow}N$  which pull back the cohomology
class of $\sigma$ onto the foliated cohomology class of $\omega$.

If we consider the trivial foliation (foliation with a single
leaf) on $M$ then as a special case we obtain the symplectic
immersion theorem due to Gromov \cite{gromov}.

For a simple application of the above theorem consider the
Euclidean manifold $\R^{2n}$ with the canonical symplectic form
$\sigma_0$. \begin{Corollary} Let $\omega$ be a foliated closed
2-form on $(M,{\mathcal F})$. Suppose rank\,$(\omega|_{T{\mathcal
F}_x})\geq r$ for all $x\in M$, for some $0\leq r\leq \dim
{\mathcal F}$.

Then there exists a foliated immersion $f:(M,{\mathcal F})\lgra
\R^{2n}$ such that $f^*\sigma_0=\omega$ on each leaf of the
foliation provided $2n\geq \dim M+ 2\dim{\mathcal
F}-r-1$.\label{coro}
\end{Corollary}

The proof of the main theorem is based on the sheaf theoretic
technique in $h$-principle and the Nash Moser implicit function
theorem due to Gromov \cite[\S 2.2, \S 2.3]{gromov}. Starting with
a formal solution $(F_0,f_0)$ as in the hypothesis of the theorem,
we embed the given manifold $(M,{\mathcal F})$ in a foliated
manifold $(M',{\mathcal F}',\omega')$, where ${\mathcal F}'$ is a
foliation with $\dim{\mathcal F}'=\dim N$ and $\omega'$ is a
2-form which restricts to a symplectic form on each leaf of the
foliation ${\mathcal F}'$. Moreover, the pullback of the form
$\omega'$ on $M$ equals $\omega$ on $T{\mathcal F}$.
We observe that the foliated symplectic forms exhibit a stability
property analogous to the stability of symplectic forms as in
Moser's theorem (\cite{wein1}, \cite{wein2}) and this observation
plays an important role in our proof.  We prove through a sequence
of propositions that there exists a foliated symplectic immersion
$f':M'\lgra N$ such that the derivative of $f'|_M$ is homotopic to
$F_0$ in Symp$^0(T{\mathcal F},TN)$.

We refer the reader to \cite{datta} for a quick review of the
terminology and the theory of topological sheaves which we shall
extensively use in this paper. For a detailed exposition on this
we refer to \cite{gromov}. We organise the paper as follows. In
Section 2 we briefly review the foliated cohomology theory. In
Section 3 we discuss the Poincar\'{e} Lemma in the context of
foliated closed forms and in Section 4 we prove an analogue of
Moser's theorem for foliated symplectic forms. The proof of the
main result of this paper is given in sections 5 through 7. \\

\section{Foliated de Rham cohomology}

Let $M$ be a smooth manifold with a regular foliation $\mathcal F$
on it. Denote the space of smooth differential $r$-forms on $M$ by
$\Omega^r(M)$. We define as in \cite{molino}, for each $r\geq
0$,\begin{center} $I^r({\mathcal F})= \{\omega{\in}\Omega^r(M) :
\omega|_{\mathcal{F}} =0\}$,\end{center} where
$\omega|_{\mathcal{F}}=0$ means that for any $x\in{M},
\omega(x)(v_1,v_2,\dots,v_r)=0$ for all $v_1,v_2,....,v_r$ in
$T_x{\mathcal F}$. Clearly, $I^r(\mathcal{F})$ is a linear
subspace of  $\Omega^r(M)$ and $I^0({\mathcal F})=\Omega^0(M) =
C^\infty(M)$. Moreover, $I({\mathcal F})=\cup_{r\geq
0}I^r(\mathcal{F})$ is a graded ideal of the deRham complex
$\Omega^*(M)$. Let $\Omega^r({\mathcal
F})=\frac{\Omega^r({\mathcal M})}{I^r({\mathcal F})}$ and $q:
\Omega^r(M){\rightarrow}\Omega^r(\mathcal{F})$ denote the quotient
map. The exterior differential operator $d :
\Omega^r(M){\rightarrow}\Omega^{r+1}(M)$ induces a morphism
$d_{\mathcal{F}}:{\Omega^r(\mathcal{F})}\lgra
{\Omega^{r+1}(\mathcal{F})}$ by $d_{\mathcal{F}}(q(\omega)) =
q(d\omega)$, since $d$ maps $I^r({\mathcal F})$ into
$I^{r+1}({\mathcal F})$. It can be easily checked that
$d_{\mathcal{F}}{\circ}d_{\mathcal{F}} = 0$ so that
$(\Omega^r({\mathcal F}),d_{\mathcal F})$ is a cochain complex
which is called the foliated de Rham complex of $(M,{\mathcal F})$\\
$$0{\rightarrow}\Omega^0(M){\rightarrow}\Omega^1(\mathcal{F})
{\rightarrow}\Omega^2(\mathcal{F}){\rightarrow}\dots{\rightarrow}\Omega^r(\mathcal{F})
{\rightarrow}\Omega^{r+1}(\mathcal{F}){\rightarrow}\dots$$

Let $T{\mathcal F}$ denote the subbundle of $TM$ consisting of all
vectors which are tangent to the leaves of the foliation and let
$T^*{\mathcal F}$ denote the dual of this bundle. Then it may be
noted that $\Omega^k({\mathcal F})$ is isomorphic to the space of
sections of the vector bundle $\Lambda^kT^*{\mathcal F}$. With
this identification we observe that $q$ is induced by the quotient
map $T^*M\lgra T^*{\mathcal F}$.

The commutative diagram below says that the quotient map $q$ is a
chain map between de Rham complex and foliated de Rham complex

$$\begin{array}{cccccccc}\begin{array}{c}0\\ \parallel \\
0\end{array}&
\begin{array}{c}\lgra\\\mbox{}\\\lgra\end{array}&
\begin{array}{c}\Omega^0(M)\\\downarrow \\\Omega^0(\mathcal F)\end{array}&
\begin{array}{c}\stackrel{d}{\lgra}\\\mbox{}\\\stackrel{d_{\mathcal F}}{\lgra}\end{array}&
\begin{array}{c}\Omega^1(M)\\ \downarrow\\\Omega^1({\mathcal F})\end{array}&
\begin{array}{c}\stackrel{d}{\lgra}\\\mbox{}\\\stackrel{d_{\mathcal F}}{\lgra}\end{array}&
\begin{array}{c}\Omega^2(M)\\ \downarrow\\\Omega^2({\mathcal F})\end{array}&
\begin{array}{c}\stackrel{d}{\lgra}\dots\\\mbox{}\\\stackrel{d_{\mathcal F}}{\lgra}\dots\end{array}\end{array}$$

The cohomology of the complex $(\Omega^*({\mathcal F}),d_{\mathcal
F})$ is defined as the {\em foliated deRham cohomology\/} of
$(M,{\mathcal F})$ and is denoted by $H^r({\mathcal F})$. In other
words,
$$H^r(\mathcal{F})=\frac{\mbox{Ker}\{d_{\mathcal{F}}:\Omega^r(\mathcal{F}){\rightarrow}\Omega^{r+1}(\mathcal{F})\}}
{\mbox{Im}\{d_{\mathcal{F}}:\Omega^{r-1}(\mathcal{F}){\rightarrow}\Omega^r(\mathcal{F})\}}.$$

It follows directly from the alternative description of the
foliated deRham complex that the foliated cohomology groups vanish
in dimensions $r>\dim{\mathcal F}$.

For foliated manifolds with single leaf the foliated de Rham
complex is the same as the ordinary de Rham complex. Hence the
foliated deRham cohomology is the same as the ordinary deRham
cohomology.

Also, if $F$ is a smooth manifold and $M$ has a trivial foliation
with leaves diffeomorphic with $F$ then $H^k(\mathcal F)\cong
H^k(F)$.

A smooth map $f:(M,\mathcal{F}){\rightarrow}(N,\mathcal{F'})$
between two foliated manifolds is called {\em foliation
preserving\/} if $f$ takes a leaf of $\mathcal{F}$ into a leaf of
$\mathcal{F'}$. Such a map $f$ induces a  chain map
$f^{\sharp}:\Omega^*({\mathcal F}'){\rightarrow}\Omega^*({\mathcal
F})$ and hence a morphism $f^*:H^*({\mathcal
F}'){\rightarrow}H^*({\mathcal F})$ in the cohomology level.

\section{Poincar\'{e} Lemma for Foliated Manifolds}
Throughout this section, $M$ will denote a smooth manifold with a
regular foliation $\mathcal F$.

\begin{Definition} {\em An $r$-form $\omega$ on $M$ is said to be {\em foliated
closed\/} if $d_{\mathcal{F}}(q\omega)=0$, that is, if $\omega$
restricts to a closed form on each leaf of the foliation.
Similarly, an $r$-form $\omega$ on $M$ is said to be {\em foliated
exact\/} if there exists an $(r-1)$-form $\varphi$ on $M$ such
that $q\omega=d_{\mathcal{F}}(q\varphi)$ which implies that
$\omega$ restricts to an exact form on each leaf of the foliation.
}\end{Definition}

Locally, every foliated closed form on any foliated space
$(M,{\mathcal F})$ is foliated exact. In fact we have the
following:

\begin{Proposition} Let $\omega$ be a foliated closed form on
$(M,\mathcal{F})$ such that $\omega$ vanishes on $T{\mathcal F}_x$
for some $x\in M$. Then there exists a local 1-form $\alpha$ such
that $\alpha$ vanishes on $T{\mathcal F}_x$ and $\omega=d\alpha$
on $T{\mathcal F}$.\label{poin1}\end{Proposition}

The proposition is a consequence of a more general result stated
below.
\begin{Proposition} Let $(M,\mathcal{F})$ be a  smooth foliated manifold and let $\pi:M'{\rightarrow}M$ be a
vector bundle over $M$. Let $\mathcal{F'}$ be the foliation on
$M'$ defined by $\mathcal{F'}={\pi}^{-1}(\mathcal{F})$. Suppose
$\omega$ is a foliated closed $k$-form on $(M',{\mathcal F}')$
such that $i^*\omega =0$ on $\mathcal{F}$, where
$i:M{\rightarrow}M'$ embeds $M$ as the zero section in $M'$. Then
there exists a neighbourhood $U$ of $M$ in $M'$ with a
$(k-1)$-form $\beta$ on $U$ such that $d\beta=\omega$ on
$\mathcal{F'}$ and
$\beta|_{i(M)}=0$.\label{poin2}\end{Proposition}

\begin{proof} Since $\pi:M'\lgra M$ is a vector bundle we
denote an element in $M'$ over $x\in M$ by $(x,v)$, where $v$ is
in the fibre over $x$. Define for each $t\in [0,1]$ a smooth map
$\rho_t:M'\lgra M'$ by $(x,v)\mapsto (x,tv)$. Then
$\rho_0(M')\subset M$ and $\rho_1=\,\mbox{id}_{M'}$. From the
definition of ${\mathcal F}'$ it is clear that each $\rho_t$ is
foliation preserving. Let $X_t$ denote the vector field along
$\rho_t$ defined by $X_t=\frac{d}{ds}\rho_s|_{s=t}$. Then
$X_t(x,v)=v$ for all $(x,v)\in M'$ so that $X_t$ is a foliated
vector field on $M'$. In particular $X_t(x,0)=0$ for all $x\in M$.

As in \cite{wein1} define an operator
$I:\Omega^k(M')\lgra\Omega^{k-1}(M')$ by
$I(\tau)=\int_0^1\rho_t^*(X_t.\tau)\,dt$, where $\tau$ is a
$k$-form on $M'$ and $X_t.\tau$ denotes the interior derivative of
$\tau$ with respect to $X_t$. First observe that $I\tau|_M=0$
since $\rho_t$ restricts to the identity map on $M$ and $X_t$
vanishes on $M$. Secondly, since $\rho_t$ is foliation preserving
and $X_t$ is a foliated vector field for each $t$, $I$ maps
$I^k({\mathcal F}')$ into itself.

Proceeding as in \cite{wein1} we integrate the relation
\begin{center}$\frac{d}{dt}(\rho_t^*\omega) =
\rho_t^*({X_t}.d\omega) + d(\rho_t^*(X_t.\omega))$.\end{center}
with respect to $t$ in $[0,1]$ to get
\begin{center}$\rho_1^*\omega -\rho_0^*\omega =
d(I\omega)+I(d\omega)$.\end{center} Noting that $i^*\omega$
vanishes on $\mathcal F$, we have $\rho^*_0\omega
=\pi^*i^*\omega=0$ on ${\mathcal F}'$. Further,  $I(d\omega) = 0$
on ${\mathcal F}'$. Hence we get $\omega = d(I\omega)$ on
$\mathcal{F'}$. The proof is completed by letting $\beta =
I(\omega)$.
\end{proof}

\begin{Corollary} Let $(M,{\mathcal F})$ and $(M',{\mathcal F}')$ be as in the above theorem and let
$\omega$ be a section of the $k$-th exterior bundle
$\Lambda^k(T^*M')$ defined over $M$. If $i^*\omega$ is a foliated
closed form on $(M,{\mathcal F})$ then $\omega$ extends to a
foliated closed form on some neighbourhood of $M$.\label{poin3}
\end{Corollary}

\begin{proof} Take any extension $\omega'$ of $\omega$ on a
neighbourhood of $M$. Defining $I$ as in the proof of Theorem
\ref{poin2} we get\begin{center}
$\omega'-\rho_0^*\omega'=I(d\omega')+dI(\omega')$.\end{center}
$\rho_0^*\omega'$ is foliated closed since
$\rho_0^*\omega'=\pi^*i^*\omega'=\pi^*i^*\omega$ and $i^*\omega$
is a foliated closed form on $(M,{\mathcal F})$. Moreover, from
the definition of $I$ it follows that $I(d\omega')|_M=0$.
Therefore taking exterior derivative of the above equation we get
$d\omega'=dI(d\omega')$ on $T{\mathcal F}'$. Define,
$\omega_1=\omega'-I(d\omega')$. Then $\omega_1$ is the desired
extension of $\omega$.
\end{proof}

\section{Stability of Foliated Symplectic Forms}

\begin{Definition} {\em A foliated closed 2-form $\omega$ on $(M,{\mathcal F})$ is said to be a
{\em foliated symplectic form\/} if $\omega$ is nondegenarate on
each leaf of ${\mathcal F}$. The foliated manifold
$(M,\mathcal{F})$ together with $\omega$ is then called a {\em
foliated symplectic manifold\/}.}\end{Definition}

\begin{Definition} {\em A vector field $X$ on a foliated manifold
$(M,{\mathcal F})$ is said to be a {\em foliated vector field\/}
if $X$ maps $M$ into $T{\mathcal F}$. The space of all foliated
vector fields on $(M,{\mathcal F})$ will be denoted by ${\mathcal
X}_{\mathcal F}$.}
\end{Definition}

Observe that a foliated symplectic form $\omega$ on $(M,{\mathcal
F})$ defines a bundle isomorphism $I_\omega:T{\mathcal F}\lgra
T^*{\mathcal F}$ which is given by the correspondence $v\mapsto
v.\omega|_{\mathcal F}$. $I_{\omega}$ induces a bijection
${\mathcal X}_{\mathcal F}\lgra\Omega^1({\mathcal F})$ which takes
a foliated vector field $X$ onto the foliated 1-form
$q(X.\omega)$.

\begin{Proposition} Suppose
$\omega_0$ and $\omega_1$ are two local foliated symplectic forms
on $(M,{\mathcal F})$ such that $\omega_1=\omega_0$ on $T{\mathcal
F}_x$ for some $x\in M$. Then there exist open neighbourhoods $U$
and $V$ of $x$ in $M$ and a foliation preserving isotopy
$\delta_t:U\lgra V$, $0\leq t\leq 1$, such that $d\delta_t(x)$ is
the identity map of $T_xM$ for all $t$, and
$\delta_1^*\omega_1=\omega_0$ on ${\mathcal
F}$.\label{moser1}\end{Proposition}

\begin{Proposition} Let $M$,
$(M',\mathcal{F'})$ be as in Proposition \ref{poin2}. Suppose
$\omega_0$ and $\omega_1$ are two foliated symplectic forms on
$M'$ such that $\omega_1=\omega_0$ on $T{\mathcal F}'|_M$. Then
there exist open neighbourhoods $U$ and $V$ of $M$ in $M'$ and a
foliation preserving isotopy $\delta_t:U\lgra V$ such that
$d\delta_t=id$ on $TM'|_M$ and $\delta_1^*\omega_1=\omega_0$ on
${\mathcal F}'$.\label{moser2}\end{Proposition}
\begin {proof}
It follows from the hypothesis that $\omega_1-\omega_0$ is a
foliated closed form which vanishes on $T{\mathcal F}'|_M$.
Therefore, by Proposition~\ref{poin2}, $\omega_1-\omega_0
=d\alpha$ on $\mathcal{F'}$, for some 1-form $\alpha$ satisfying
$\alpha|_M=0$. For $0\leq t\leq 1$ we define a family of foliated
closed forms $\omega_t$ by $\omega_t = \omega_0 + t d\alpha$.
Since $\omega_1=\omega_0$ on $T{\mathcal F}'|_M$, each $\omega_t$
restricts to a foliated symplectic form on some neighbourhood $U$
of $M$ in $M'$. For each $t\in[0,1]$, define a foliated vector
field $X_t$ by $X_t={I_{\omega_t}}^{-1}(-q(\alpha))$ so that
$X_t.\omega_t=-\alpha$ on $T{\mathcal F}'$. Let $\delta_t$,
$t\in[0,1]$ be the one parameter family of diffeomorphisms defined
on some open neighbourhood of $M$ such that
$\delta_0=\mbox{\,id}_M$ and $\frac{d\delta_t}{dt}=X_t$. Since
$\alpha=0$ on $T{\mathcal F}'|_M$, it follows that
$\delta_t|_M=id$. Moreover, since $X_t$ is a foliated vector
field, $\{\delta_t\}$ is a foliation preserving diffeotopy.

Now, consider the identity
\begin{center}$\frac{d}{dt}(\delta_t^*\omega_t)=
{\delta_t}^*(\frac{d}{dt}{\omega_t}+X_t.
d\omega_t+d(X_t.\omega_t))$.\end{center} Since $\omega_t$ is a
foliated closed form, $X_t.d\omega_t=0$ on ${T\mathcal F}'$ we
obtain from the above relation
$\frac{d}{dt}q(\phi_t^*\omega_t)=0$. Hence
$\delta_t^*\omega_t=\delta_0^*\omega_0=\omega_0$ on $T{\mathcal
F}'$. In particular $\delta_1^*\omega_1=\omega_0$ on $T{\mathcal
F}'$. Also, it is easy to see that $\delta_t$ fixes $M$ pointwise
and therefore
$\frac{d}{dt}(d\delta_t)(x)=d\frac{d}{dt}\delta_t(x)=d(X_t(x))=0$
for all $x\in M$.\end {proof}

We end this section with the following result.

\begin{Theorem} Let $M$ be a closed manifold and let $\mathcal F$ be a regular foliation on
$M$. Let $\{\omega_t\}_{t\in [0,1]}$  be a smooth family of
foliated symplectic forms on $M$ such that
$\omega_t=\omega_0+{d}\alpha_t$ on $\mathcal{F}$, where $\alpha_t$
is a smooth family of $1$-forms. Then there exists a smooth
foliated diffeotopy $\delta_t:{M}\rightarrow {M},\
\delta_0=\mbox{id}$, such that ${\delta_t}^*(\omega_t)=\omega_0$
on $\mathcal F$ (that is, ${\delta_t}^*(\omega_t)=\omega_0$ in
$\Omega^2({\mathcal F})$).\end{Theorem}

\section{Construction of an extension}

In the subsequent discussion, $(N,\sigma)$ will denote a
symplectic manifold and $(M,\mathcal F)$ will denote a foliated
manifold with a foliated closed 2-form $\omega$.

Let $(F_0,f_0):TM\lgra TN$ be a bundle homomorphism which
satisfies the hypothesis of Theorem~\ref{main}:
\begin{enumerate}\item $F_0|_{T{\mathcal F}}:T{\mathcal F}\lgra
TN$ is a monomorphism and $F_0^*\sigma=\omega$ on $T{\mathcal F}$;
\item $f_0:M\lgra N$ is a continuous map such that the foliated
cohomology class of $f_0^*\sigma$ is the same as that of $\omega$.
\end{enumerate}

\begin{Proposition} There is a foliated
symplectic manifold $(M',{\mathcal F}',\omega')$ and a foliation
preserving embedding $i:(M,{\mathcal F})\lgra (M',{\mathcal F}')$
such that $i^*\omega'=\omega$ on $T{\mathcal F}$, where $i$ is the
embedding of $M$ in $M'$.

Further, $(F_0,f_0)$ extends to a bundle homomorphism
$(F_0',f_0'):TM'\lgra TN$ which satisfies the following
properties:
\begin{enumerate}\item $F'_0|_{T{\mathcal F}'}:T{\mathcal F}'\lgra
TN$ is a monomorphism and ${F'_0}^*\sigma=\omega'$ on $T{\mathcal
F}'$; \item $f_0:M'\lgra N$ is a continuous map such that the
foliated cohomology class of ${f'_0}^*\sigma$ is the same as that
of $\omega'$.
\end{enumerate}\label{extension}
\end{Proposition}

{\em Proof\/.} Consider the quotient bundle
$\pi:f_0^*TN/T{\mathcal F}\lgra M$ and denote the total space of
the bundle by $X$. $\dim X=\dim M+(\dim N-\dim {\mathcal
F})=\mbox{\,codim\,} {\mathcal F}+\dim N$. Let ${\mathcal F}'$ be
the foliation on $X$ defined by the map $\pi$, that is ${\mathcal
F}'=\pi^{-1}({\mathcal F})$. Clearly, codim\,${\mathcal F}'=$
codim\,${\mathcal F}$. Hence $\dim {\mathcal F}'=\dim N$.

$M$ is embedded in $X$ as the zero-section of $\pi$; consequently
$TM$ is canonically embedded in $TX|_M=TM\oplus X$. Let
$\tilde{F}$ be any extension of $F_0$ to a bundle morphism
$TX|_M\lgra f_0^*TN$ such that it maps $X$ isomorphically onto a
complementary subbundle of $df_0(T{\mathcal F})$ in $f_0^*TN$.
Define $\bar{F}_0=q\circ \tilde{F}$, where $q:f_0^*TN\lgra TN$ is
the canonical bundle morphism. By the above construction
$\bar{F}_0|_{T{\mathcal F}'}$ is a bundle isomorphism over $M$.
Therefore, $\bar{F}_0^*\sigma$ is non-degenerate on $T{\mathcal
F}'|_M$.

Since $\bar{F}_0^*\sigma$ restricts to a foliated closed form on
$M$, by Corollary~\ref{poin3} it extends to a foliated closed (and
hence a foliated symplectic) form $\omega'$ on a neighbourhood of
$M$ in $(X,{\mathcal F}')$. Then, following
Proposition~\ref{moser2}, we can show that $\bar{F}_0$ extends to
a bundle morphism $F_0'$ on an open neighbourhood $M'$ of $M$ such
that ${F'_0}^*\sigma=\omega'$ on $T{\mathcal F}'$.

The inclusion $i:M\hookrightarrow M'$ induces a morphism
$\Omega^2({\mathcal F}')\lgra \Omega^2(\mathcal F)$ which takes
the class of $\omega'$ in $\Omega^2({\mathcal F}')$ onto the class
of $\omega$ in $\Omega^2({\mathcal F})$. The induced map
$i^*:H^2({\mathcal F}')\lgra H^2(\mathcal F)$ in the cohomology
level, therefore, maps the foliated cohomology class of $\omega'$
onto that of $\omega$.

If $f_0'$ denotes the underlying map of $F_0'$ then it is an
extension of $f_0$. Since $i^*$ is an isomorphism, it follows that
$f_0'$ pulls back the deRham cohomology class of $\sigma$ onto the
foliated cohomology class of $\omega'$.\qed\\


\section{Sheaf of foliated Symplectic immersions}

\begin{Definition}{\em A foliated immersion $f:(M,{\mathcal F},\omega)\lgra (N,\sigma)$ is said to be a
{\em foliated symplectic immersion\/} if the restriction of
$f^*\sigma$ is same as the restriction of $\omega$ on each leaf of
the foliation.}
\end{Definition}

Let $M',{\mathcal F}'$ and $\omega'$ be as defined in the previous
section. Let ${\mathcal S}_{{\mathcal F}'}$ denote the sheaf of
foliated immersions $f:(M',{\mathcal F}')\lgra N$ which satisfy
$f^*\sigma=\omega'$ on ${\mathcal F}'$. The topology of the sheaf
comes from the $C^\infty$ compact open topology on
$C^\infty(M,N)$.

Let ${\mathcal R}_{{\mathcal F}'}\subset J^1(M',N)$ be the space
of 1-jets of foliated symplectic immersions of
$(M',\mathcal{F'},\omega')$ in $(N,\sigma)$. We endow the sheaf of
sections of $J^1(M',N)$ with the $C^0$ compact open topology.
${\Psi}_{\mathcal{F'}}$ will denote the subsheaf of sections of
the 1-jet bundle whose images lie in $\mathcal{R}_{\mathcal{F'}}$.
Note that $\Psi_{{\mathcal F}'}(M')$ can be identified with the
space Symp$^0(T{\mathcal F}',TN)$ defined in the introduction.

The sheaf $\Psi_{{\mathcal F}'}|_M$ is flexible
\cite[1.4.2(A$'$)]{gromov}. Moreover, we have the following.
\begin{Proposition} The $1$-jet map $j^1:{\mathcal S}_{{\mathcal
F}'}(x)\rightarrow \Psi_{{\mathcal F}'}(x)$ is a weak homotopy
equivalence for all $x\in M'$.\label{whe}\end{Proposition}

\begin{proof} In view of the discussion in \cite[2.3.2(D),(D$')$]{gromov} it
is enough to show that an infinitesimal solution of ${\mathcal
R}_{{\mathcal F}'}$ can be homotoped to a local solution. Let
$j^1_f(x)\in {\mathcal R}_{{\mathcal F}'}$ so that
$df_x|T{\mathcal F}'_x$ is an injective linear map and $f^*\sigma
= \omega'$ on $T{\mathcal F}'_x$.

Since $\dim {\mathcal F}'=\dim N$ and $df(x)$ is injective on
$T{\mathcal F}'(x)$ it follows that $f$ is a foliated immersion on
a neighbourhood of $x$ and $f^*\sigma$ is non-degenerate on
$T{\mathcal F}'$. Let $\bar{\omega}=f^*\sigma$. $\bar{\omega}$ is
a foliated symplectic form on a neighbourhood of $x$ such that
$\bar{\omega}_x=\omega'_x$ on $T{\mathcal F}'_x$. By applying
Proposition~\ref{moser1} we get a local diffeotopy $\delta_t$ on
Op$(x)$ in $M'$ such that $d\delta_t=\mbox{\,id}$ at $x$ and
$\delta_1^*f^*\sigma=\delta_1^*\bar{\omega}=\omega'$ on
$T{\mathcal F}'_x$. Let $f_t=f\circ \delta_t$, $0\leq t\leq 1$.
$f_1$ is a local solution of ${\mathcal R}_{{\mathcal F}'}$ and
$j^1_{f_t}(x)=j^1_{f}(x)\in{\mathcal R}_{{\mathcal F}'}$ for all
$t$. Thus we have proved that an infinitesimal solution can be
homotoped to a local solution of ${\mathcal R}_{{\mathcal F}'}$
and this completes the proof.
\end{proof}

It is important to note that ${\mathcal S}_{{\mathcal F}'}$ is not
microflexible \cite[3.4.1(B)]{gromov} and therefore we can not
apply the sheaf theoretic technique directly to it. But it is
possible to find an associated sheaf which has this desired
property apart from having the same local weak homotopy type as
${\mathcal S}_{{\mathcal F}'}$.

To define the associated sheaf we start with the bundle map
$(F'_0,f'_0):TM'\lgra TN$ that we constructed in
Proposition~\ref{extension}.

Consider the product manifold $M'\times N$ with the product form
$\hat{\omega}=p_2^*\sigma-p_1^*\omega'$, where $p_1$ and $p_2$ are
the projection maps from $M'\times N$ onto the first and the
second factors respectively. Let $\hat{\mathcal F}$ denote the
foliation on $M'\times N$ induced ${\mathcal F}'$ by the first
projection map $p_1$. It is easy to see that $\hat{\omega}$ is a
foliated symplectic form on $(M'\times N, \hat{\mathcal F})$.

If $f_0'$ is as above then $g_0=(1,f_0'):(M',{\mathcal F}')\lgra
(M'\times N,\hat{\mathcal F})$ is a foliation preserving embedding
and therefore it induces a map $g_0^*: H^2(\hat{\mathcal F})\lgra
H^2({\mathcal F}')$ in the foliated cohomology level.
It is easily seen that $g_0^*([\hat{\omega}]_{\hat{\mathcal
F}})=0$ in the foliated cohomology group $H^2({\mathcal F}')$.
Now, there exists a neighbourhood $Y$ of Image\,$g_0$ such that
Image\,$g_0$ is a strong deformation retract of $Y$; further the
deformation retraction is foliation preserving. Consequently,
$\hat{\omega}$ is a foliated exact form on $(Y,\hat{\mathcal F})$.
Hence there exists a 1-form $\tau$ on $Y$ such that
$\hat{\omega}=d\tau$ on $T{\hat{\mathcal F}}$.

Suppose, $f:(M',{\mathcal F}')\lgra N$ is a foliated immersion
such that $f^*\sigma$ and $\omega'$ define the same foliated form
in $\Omega^2({\mathcal F}')$, and suppose the graph map $g=(1,f)$
has its image contained in $Y$. Then observe that $g$ is foliation
preserving and $g^*\tau$ is a foliated closed form on
$(M,{\mathcal F}')$.

Denote by $\Gamma^\infty(Y)$, the space of sections $M'\lgra
M'\times N$ whose images lie in $Y$. Let
$\mathcal{E}_{\mathcal{F'}}$ denote the sheaf of all those pairs
$(g,\varphi)$ in $\Gamma^\infty(Y)\times C^\infty(M')$ which
satisfy the following conditions:
\begin{enumerate}\item $p_2 g:(M',{\mathcal F}')\lgra N$ is a
foliated immersion, so that $g$ is foliation preserving, and \item
$g^*\tau+d\varphi=0$ on $T{\mathcal F}'$.\end{enumerate}

Observe that there is a canonical map $\pi:\mathcal{E}_{{\mathcal
F}'}\lgra {\mathcal S}_{{\mathcal F}'}$ which maps $(g,\varphi)$
onto $p_2 g$.

\begin{Proposition} The topological sheaves $\mathcal{E}_{{\mathcal
F}'}$ and ${\mathcal S}_{{\mathcal F}'}$ have the same local weak
homotopy type.

Consequently, $j^1\circ \pi: \mathcal{E}_{{\mathcal
F}'}(x)\rightarrow \Psi_{{\mathcal F}'}(x)$ is a weak homotopy
equivalence for all $x\in M'$.\label{local}\end{Proposition}
\begin{proof} Take any two elements $(g, \varphi_1)$ and
$(g,\varphi_2)$ in $\mathcal{E}_{\mathcal{F'}}(x)$ over some $f\in
{\mathcal S}_{{\mathcal F}'}(x)$, then $d(\varphi_1- \varphi_2)=0$
on $\mathcal{F'}$ in some foliated neighbourhood of $x \in M'$
which implies that $(\varphi_1-\varphi_2)=$ constant on each
plaque of $\mathcal{F'}$. If $\mathcal{F'}$ is of codimension $k$,
we get $\mathcal{E}_{\mathcal{F'}}(x)=
\mathcal{S}_{\mathcal{F'}}(x){ \times }\mathcal{C}^\infty(0)$,
where $\mathcal{C}^\infty(0)$ is the space of germs of real valued
functions on $\R^k$ at $0$. Since $\mathcal{C}^\infty(0)$ is
contractible, this completes the proof.\end{proof}

\begin{Lemma} The sheaf ${\mathcal E}_{{\mathcal F}'}$ is
microflexible.\label{micro}
\end{Lemma}
\begin{proof} Define a first order differential operator
\begin{center}${\mathcal D}: \Gamma^\infty(Y)\times C^\infty(M')
\lgra \Omega^1({\mathcal F}')$\end{center}by \hspace{4cm}$(g,\phi)
\mapsto (g^*\tau+d\phi)|_{T{\mathcal F}'}.$\\

Let $L$ denote the linearisation of ${\mathcal D}$ at $(g,\phi)$.
Then
\begin{center}$\begin{array}{rcl}L(\partial,\tilde{\phi}) & = &
g^*(\partial.\hat{\omega}+d(\partial.\tau)+d\tilde{\phi})|_{T{\mathcal
F}'},\end{array} $\end{center} where $\partial$ is a smooth vector
field on $N$ along $g$ and $\tilde{\phi}$ is a smooth function on
$M$.

$L$ is right invertible if given any 1-form
$\tilde{\omega}\in\Omega^1({\mathcal F}')$ we can solve the
following system of equations:
\begin{center}$g^*(\partial.\hat{\omega})|_{T{\mathcal F}'} = \tilde{\omega}$\\
$\partial.\tau+ \tilde{\phi} = 0 $\end{center}

Note that $\hat{\omega}$ is a foliated symplectic form on
$(\hat{M},\hat{\mathcal F})$; hence, if $g:M'\lgra Y\subset
M'\times N$ is a foliation preserving section, then the map
$\partial \mapsto g^*(\partial.\hat{\omega})|_{T{\mathcal F}'}$ is
an epimorphism. Consequently, for such a $g$ we can solve the
first equation for $\partial$, and then take
$\tilde{\phi}=-\partial.\tau$.

Thus, the operator ${\mathcal D}$ is infinitesimally invertible on
all those $(g,\phi)$ for which $p_2 g:(M',{\mathcal F'})\lgra
(Y,\hat{\mathcal F})$ is a foliation preserving immersion.

Further, the foliation preserving immersions are solutions to some
first order open differential relation.

Hence ${\mathcal E}_{{\mathcal F}'}$ is a microflexible sheaf
(\cite[2.3.2]{gromov}).\end{proof}

\section{Proof of Theorem \ref{main}}

We shall first prove that ${\mathcal E}_{{\mathcal F}'}|_M$ is
flexible. Once we show this we can conclude with the Sheaf
Homomorphism Theorem that $j^1\circ\pi:{\mathcal E}_{{\mathcal
F}'}|_M\lgra \Psi_{{\mathcal F}'}|_M$ is a weak homotopy
equivalence.

Recall a result on continuous sheaves from \cite[2.2.3]{gromov}.

\begin{Theorem}  Let $\Phi$ be a microflexible sheaf over a manifold $V$
and let a submanifold $V_0\subset V$ of positive codimension be
sharply movable by acting diffeotopies. Then the sheaf
$\Phi_0=\Phi|_V$ is a flexible sheaf. \label{gromov}\end{Theorem}

Consider the pseudogroup ${\mathcal D}$ consisting of local
diffeotopies $\delta_t$ on $M'$ that preserve the foliation
${\mathcal F}'$ and at the same time preserve the class of
$\omega'$ in $\Omega^2({\mathcal F}')$. Clearly ${\mathcal D}$
acts on the sheaf ${\mathcal S}_{{\mathcal F}}'$. Such
diffeotopies may be obtained by integrating a foliated vector
field $\partial$ such that $\partial.\omega'$ is a foliated closed
form on $(M',{\mathcal F}')$. Indeed if $\partial$ is a foliated
vector field then its restriction to each leaf is a vector field
on the leaf. Thus it integrates to a foliation preserving
diffeomorphism on $M'$. Also, observe that ${\mathcal L}_\partial
\omega'=\partial.d\omega'+d(\partial.\omega')=0$ on $T{\mathcal
F}'$. Hence $\delta_t^*\omega'=\omega'$ in $\Omega^2({\mathcal
F}')$.

We shall show that there exists a pseudosubgroup of ${\mathcal D}$
which acts on ${\mathcal E}_{{\mathcal F}'}$ and sharply moves $M$
in $M'$. As we have already proved microflexibility of ${\mathcal
E}_{{\mathcal F}'}$, the flexibility of ${\mathcal E}_{{\mathcal
F}'}|_M$ will follow from Theorem~\ref{gromov}.

\begin{Definition}{\em A foliation preserving diffeotopy $\delta_t$
of $(M',{\mathcal F}')$ is called {\em exact foliation preserving
diffeotopy\/} if there exists a 1-parameter family of 0-forms
$\alpha_t$ on $M'$ such that $\delta'_t.\omega'=d\alpha_t$ on
${\mathcal F}'$. In particular, $\delta'_t.\omega'$ is a foliated
exact form for each $t$.

If $\alpha_t$  can be chosen to be identically zero on the maximal
open subset where $\delta_t$  is constant then such a diffeotopy
is called a }strictly exact foliation preserving
diffeotopy.\end{Definition}

\begin{Lemma} The strictly exact foliation preserving diffeotopies of $M'$
act on $\mathcal{E}_{\mathcal{F'}}$.\label{action}\end{Lemma}

{\em Proof\/.} Let $\delta_t$ be a strictly exact foliation
preserving diffeotopy on $M'$. We define a diffeotopy
$\bar{\delta_t}$ on $M'{\times}N$ by $\bar{\delta_t}(x,y) =
(\delta_t(x),y)$, where $x\in{M'}$  and $y\in{N}$. Recall that
$M'{\times}N$ has the product foliation $\hat{\mathcal{F}}$ and
$\bar{\delta_t}(\hat{\mathcal F})\subseteq\hat{\mathcal F}$.
Moreover, $\bar{\delta'_t}.(\sigma-\omega')$ is foliated exact
(since ${\delta'_t}.\omega'$ is foliated exact).

Proceeding exactly as in \cite[3.4.1]{gromov}, let $\alpha_t$ be a
smooth family of 0-forms on $M'{\times}N$ satisfying
$\bar{\delta'_t}.(\sigma-\omega') = d{\alpha_t}$  on
$\hat{\mathcal{F}}$. Since,
\begin{center}$\frac{d}{dt}(\bar{\delta}^*_t{\tau})=
\bar{\delta}^*_t(d(\bar{\delta'_t}.\tau) +
\bar{\delta'_t}.d{\tau})=\bar{\delta}^*_t(d(\bar{\delta'_t}.\tau)+\bar{\delta'_t}.(\sigma-\omega'))$\end{center}
we have
\begin{equation}
 \frac{d}{dt}(\bar{\delta}^*_t{\tau}) = d\bar{\delta^*_t}(\bar{\delta}'_t.\tau + \alpha_t)\ \mbox{ on } \hat{\mathcal F}.
\end{equation}
Integrating this equation with respect to $t$ we obtain
$\bar{\delta_t^*}\tau = \tau + d{\varphi_t}$  on
$\hat{\mathcal{F}}$, where $\varphi_t$ is a family of functions on
$M'$.

Let $(g,\varphi)$ be a section in ${\mathcal E}_{{\mathcal F}'}$
and let $\delta_t$ be as above. Define the action as in
\cite{gromov} by $\delta_t.(g,\varphi) = (g',\varphi')$ , where
$g'= \bar{\delta_t}g\delta_t^{-1}$ and
${\varphi}'=(\delta_t^{-1})^*( \varphi  + g^*{\varphi_t})$.\qed

\begin{Lemma} The strictly exact foliation preserving diffeotopies of
$(M',\omega',{\mathcal F}')$ sharply move
$M$.\label{sharp}\end{Lemma}

{\em Proof\/.} Suppose  $S$  is a closed hypersurface which lies
in a small open set $U$ of $M$. Take a vector
$\partial_0{\in}{T_{x_0}{\mathcal{F'}}}$ transversal to $U$ in $M$
and extend it to a foliated vector field $\partial =
I^{-1}_{\omega'}(q(dH))$ which is transversal to $U$ (provided we
take $U$ sufficiently small). The isotopy $\{\delta_t\}$ defined
by $\partial$ is foliation preserving. Take $S_{\epsilon} =
\bigcup_{t\in[0,\epsilon]}{\delta_t(S)}{\subseteq}M'$ and
$a{\in}C^\infty(M')$ so that $a$ vanishes outside a neighbourhood
of $\mathcal{OP}(S_{\epsilon})$ and $a=1$ on a smaller
neighbourhood of $S_\epsilon$. The diffeotopy corresponding to
$I^{-1}_{\omega'}(q(d(aH)))$ can move $S$ (along the foliation) as
sharply as we want.\qed

Combining the results obtained in Lemma~\ref{micro},
Lemma~\ref{action} and Lemma~\ref{sharp} we conclude by
Theorem~\ref{gromov} that
\begin{Proposition}The sheaf $\mathcal{E}_{\mathcal{F'}}|_{M}$  is
flexible.\label{flexible}
\end{Proposition}

{\em Proof of Theorem~\ref{main}}. It now follows from
Proposition~\ref{local} and Proposition~\ref{flexible}  (as a
consequence of Sheaf Homomorphism Theorem (\cite[2.2.1]{gromov}))
that the composition $\mathcal{E}_{\mathcal{F'}}|_M\lgra{\mathcal
S}_{\mathcal{F'}}|_M\stackrel{d}{\lgra}\Psi_{{\mathcal F}'}|_M$ is
a weak homotopy equivalence. This implies that there is an $f'\in
{\mathcal S}_{\mathcal{F'}}|_M$ such that $Df'$ is homotopic to
$F'_0$ in Symp$^0(T{\mathcal F}',TN)$. Moreover, the homotopy
between $f_0'$ and $f'$ can be made to lie in an arbitrary $C^0$
neighbourhood of $f_0'$. If $f$ denotes the restriction of $f'$ to
$M$, then $f\in {\mathcal S}_{\mathcal F}$ and $Df$ is homotopic
to $F_0$ in the space Symp$^0(T{\mathcal F},TN)$. This proves
Theorem~\ref{main}.\qed\\

{\em Proof of Corollary~\ref{coro}}. Let $\omega_0$ be a linear
$2$-form on $\R^m$ of rank $r$. Then $r$ is an even integer, say
$r=2k$, such that $\omega_0^k\neq 0$ and $\omega_0^{k+1}=0$.
$\omega_0$ can be extended to a symplectic form $\bar{\omega}_0$
on $\R^{m}\times \R^{m-2k}=\R^{2(m-k)}$. Observe that an injective
linear map $\bar{F}:\R^{2(m-k)}\lgra \R^{2n}$ satisfying
$\bar{F}^*\sigma_0=\bar{\omega}_0$ restricts to an injective
linear map $F:\R^m\lgra \R^{2n}$ satisfying
$F^*\sigma_0=\omega_0$. The space of injective linear maps
$\bar{F}:\R^{2(m-k)}\lgra \R^{2n}$ satisfying 
$\bar{F}^*\sigma_0=\bar{\omega}_0$ has the same homotopy type as
the Stieffel manifold $V_{m-k} (\C^n)$ and it is known that
$V_{m-k} (\C^n)$ is $2n-2(m-k)=2n-2m+r$ connected.

Now let $\omega$ be a foliated closed $2$-form on $(M,{\mathcal
F})$, and suppose that rank $(\omega_x|_{T{\mathcal F}_x})\geq r$
for all $x\in M$ for some $0\leq r\leq m$. In view of Theorem
~\ref{main} it is enough to show the existence a monomorphism
$F:T{\mathcal F}\lgra \R^{2n}$ such that $F^*\sigma_0=\omega$. It
follows from the above discussion that such a bundle map exists if
$2n-2\dim{\mathcal F}+r\geq \dim M-1$. This completes the proof of
the corollary.

\end{document}